\documentclass{amsart}
\usepackage{amsthm}
\usepackage{amsfonts}
\usepackage{amssymb}
\usepackage{amsmath}
\usepackage{nccmath}
\usepackage{mathtools}
\usepackage{mathrsfs}
\usepackage{setspace}
\newtheorem{theorem}{Theorem}[section]

\newtheorem{cor}{Corollary}[section]

\theoremstyle{definition}
\newtheorem{defn}{Definition}
\theoremstyle{definition}

\theoremstyle{remark}

\theoremstyle{proposition}
\newtheorem{prop}{Proposition}[section]

\begin{document}
\title[Poletsky-Stessin Hardy Spaces in $\mathbb{C}^{n}$]
 {Poletsky-Stessin Hardy Spaces on Complex Ellipsoids in $\mathbb{C}^{n}$ }

\author{S\.{I}bel \c{S}ah\.{I}n }

\address{Faculty of Engineering and Natural Sciences, Sabanc{\i} University}

\email{sahinsibel@sabanciuniv.edu}


\date{\today}
\dedicatory{Dedicated to Prof.Dr. Ayd{\i}n Aytuna on the occasion of
his 65th birthday}


\begin{abstract}
We study Poletsky-Stessin Hardy spaces on complex ellipsoids in
$\mathbb{C}^{n}$. Different from one variable case, classical Hardy
spaces are strictly contained in Poletsky-Stessin Hardy spaces on
complex ellipsoids so boundary values are not automatically obtained
in this case. We have showed that functions belonging to
Poletsky-Stessin Hardy spaces have boundary values and they can be
approached through admissible approach regions in the complex
ellipsoid case. Moreover, we have obtained that polynomials are
dense in these spaces. We also considered the composition operators
acting on Poletsky-Stessin Hardy spaces on complex ellipsoids and
gave conditions for their boundedness and compactness.
\end{abstract}

\maketitle
\section*{Introduction}
The aim of this paper is to study the behavior of Hardy spaces
introduced by Poletsky-Stessin in \cite{pol} in the case of complex
ellipsoids, $\mathbb{B}^{\textbf{p}}$. Unlike the one variable case,
for $n>1$ Poletsky-Stessin Hardy spaces on complex ellipsoids
strictly contain the classical Hardy spaces
$H^{p}(\mathbb{B}^{\textbf{p}})$. Hence, in this case we do not
inherit the existence of boundary values from the classical theory.
In this paper, we show the existence of boundary values through
admissible approach regions. Moreover, we obtain a polynomial
approximation as in the classical Hardy spaces and we also consider
the boundedness and compactness properties of composition operators
acting on Poletsky-Stessin Hardy spaces of complex ellipsoid.\\ The
organization of this paper is as follows: In Section 1, we recall
the classical Hardy spaces given in \cite{stein} and review the
construction of the Poletsky-Stessin Hardy spaces
$H^{p}_{\phi}(\Omega)$, for a hyperconvex domain $\Omega$ and a
continuous, negative, plurisubharmonic exhaustion function $\phi$.
The main results of this study are given in the following sections,
in Section 2 we first examine the existence of radial limit values
for Poletsky-Stessin Hardy spaces
$H^{p}_{u}(\mathbb{B}^{\textbf{p}})$. In addition, we will give a
discussion of comparison between classical Hardy spaces and
Poletsky-Stessin Hardy classes. Then, we consider a generalization
of a method given by Stein and using this rather general method in
the case of complex ellipsoid with Cauchy-Fantappie kernel, we show
the existence of boundary values through admissible approach
regions. Moreover, we will give a brief discussion about the
relation between the admissible approach regions and Kobayashi
approach regions given by the invariant Kobayashi-Royden metric. In
this section we also show that polynomials are dense in the
Poletsky-Stessin Hardy spaces $H^{p}_{u}(\mathbb{B}^{\textbf{p}})$.
Finally, in Section 3 we consider the composition operators, with
holomorphic symbols, acting on $H^{p}_{u}(\mathbb{B}^{\textbf{p}})$
and give the necessary and sufficient conditions for the boundedness
and compactness of these operators.
\section {Poletsky-Stessin Hardy Spaces on Complex Ellipsoids}
In this section we will give the preliminary definitions and some
important results that we will use throughout this paper. Before
proceeding with Poletsky-Stessin Hardy spaces let us first recall
the classical Hardy spaces given by \cite{stein}. Let $\Omega$ be a
smoothly bounded, hyperconvex domain in $\mathbb{C}^{n}$ and
$\lambda$ be a characterizing function for $\Omega$ which is defined
in a neighborhood of $\overline{\Omega}$ i.e. $\lambda$ is smooth ,
$\lambda(x)<0$ if and only if $x\in \Omega$, $\partial
\Omega=\{\lambda(x)=0\}$ and $|\nabla\lambda(x)|>0$ if $x\in\partial
\Omega$. (The last condition is equivalent to
$\frac{\partial\lambda}{\partial\nu_{x}}>0$ where $\nu_{x}$ is the
outward normal at x.). Let $\Omega_{r}=\{z:\lambda(z)<r:r<0\}$ and
$\partial \Omega_{r}=\{z:\lambda(z)=r\}$.\\In \cite{stein}, E.M.
Stein defines the classical Hardy spaces $H^{p}(\Omega)$ as:
\begin{equation*}
H^{p}(\Omega)\doteq\{f|\quad f\quad\textit{holomorphic in}\quad
\Omega,\quad \sup_{r<0}\int_{\partial
\Omega_{r}}|f|^{p}d\sigma_{r}<\infty \}
\end{equation*} where $d\sigma_{r}$ is the surface area
measure induced by the characterizing function $\lambda$ on
$\partial \Omega_{r}$. This space is equipped with the norm
\begin{equation*}
\|f\|_{p}^{p}=\sup_{r<0}\int_{\partial
\Omega_{r}}|f|^{p}d\sigma_{r}.
\end{equation*}
The space $H^{p}(\Omega)$ does not depend on the characterizing
function $\lambda$ used to define $\Omega$ and one gets equivalent
norms for different characterizing functions. In \cite{pol},
Poletsky \& Stessin introduced new Hardy type classes of holomorphic
functions on hyperconvex domains in $\mathbb{C}^{n}$. Before
defining these new classes let us first give some preliminary
definitions. Let $\varphi:\Omega\rightarrow [-\infty,0)$ be a
negative, continuous, plurisubharmonic exhaustion function for
$\Omega$. Following \cite{de1} we define the pseudoball:
\begin{equation*}\label{eq:pseudoball}
B(r)=\{z\in\Omega:\varphi(z)<r\}\quad ,\quad r\in[-\infty,0),
\end{equation*}
and pseudosphere:
\begin{equation*}\label{eq:pseudosphere}
S(r)=\{z\in\Omega:\varphi(z)=r\}\quad ,\quad  r\in[-\infty,0),
\end{equation*}
and set
\begin{equation*}
\varphi_{r}(z)= \max\{\varphi(z),r\}\quad ,\quad r\in(-\infty,0).
\end{equation*} \\ In \cite{de1}, Demailly introduced the Monge-Amp\`ere
measures in the sense of currents as :
\begin{equation*}\label{eq:mameasure}
\mu_{\varphi,r}=(dd^{c}\varphi_{r})^{n}-\chi_{\Omega\setminus
B(r)}(dd^{c}\varphi)^{n}\quad r\in(-\infty,0).
\end{equation*}\\It is clear from the definition that these measures are supported on
$S(r)$. Demailly in \cite{de2}, proved the so-called Lelong-Jensen
Formula which we use throughout the sequel. Lelong-Jensen Formula is
stated as follows:
\begin{theorem}
Let $r<0$ and $\phi$ be a plurisubharmonic function on $\Omega$ then
for any negative, continuous, plurisubharmonic exhaustion function
$u$
\begin{equation}\label{eq:leljen}
\int_{S_{u}(r)}\phi d\mu_{u,r}-\int_{B_{u}(r)}\phi
(dd^{c}u)^{n}=\int_{B_{u}(r)}(r-u)dd^{c}\phi(dd^{c}u)^{n-1}.
\end{equation}
\end{theorem}
The next theorem gives us the comparison between Poletsky-Stessin
Hardy spaces and the classical Hardy spaces:
\begin{theorem}
Suppose that $\Omega$ is a smoothly bounded, hyperconvex domain with
a plurisubharmonic characterizing function $\rho$. Then
$H^{p}(\Omega)\subseteq H^{p}_{\rho}(\Omega)$, $1\leq p<\infty$.
\end{theorem}
\begin{proof}
Since $\rho$ is a smooth function we have the Monge-Amp\`ere measure
$d\mu_{\rho,r}=d^{c}\rho\wedge(dd^{c}\rho^{n-1})|_{S(r)}$
(\cite{de1}, Proposition 3.3) and the surface area measure induced
by $\rho$ is $d\sigma=d^{c}\rho\wedge(dd^{c}|z|^{2})^{n-1}|_{S(r)}$
(\cite{range}, Corollary 3.5). These are both $(2n-1)$-dim
differential forms on the $(2n-1)$-dim manifold so we have
$d\mu_{\rho,r}=c(z)d\sigma(z)$. In a neighborhood of
$\overline{\Omega}$, $\rho$ is smooth and
$\Omega\subset\subset\mathbb{C}^{n}$ so $c(z)$ is a bounded
function. Hence,
\begin{equation*}
\int_{S(r)}\phi d\mu_{\rho,r}=\int_{S(r)}\phi(z)c(z)d\sigma(z)\leq
K\int_{S(r)}\phi(z)d\sigma(z)
\end{equation*}Thus, we have $H^{p}(\Omega)\subseteq
H^{p}_{\rho}(\Omega)$.
\end{proof}

One of the main concerns of this study is to understand the boundary
behavior of Poletsky-Stessin Hardy spaces. For this we also need
boundary measures which were introduced by Demailly in \cite{de2}.
Now let $\varphi:\Omega\rightarrow[-\infty,0)$ be a continuous,
plurisubharmonic exhaustion for $\Omega$ and suppose that the total
Monge-Amp\`ere mass is finite that is, we assume that
\begin{equation}
MA(\varphi)=\int_{\Omega}(dd^{c}\varphi)^{n}<\infty.
\end{equation}
Then as $r$ approaches to 0, $\mu_{\varphi,r}$ converges to a
positive measure $\mu_{\varphi}$ weak*-ly on $\Omega$ with total
mass $\int_{\Omega}(dd^{c}\varphi)^{n}$ and supported on
$\partial\Omega$. This measure $\mu_{\varphi}$ is called the
\textbf{Monge-Amp\`ere
measure on the boundary associated with the exhaustion $\varphi$}.\\
Now we can introduce the Poletsky-Stessin Hardy classes, which will
be our main focus throughout this study. In \cite{pol}, Poletsky \&
Stessin gave the definition of new Hardy spaces using Monge-Amp\'ere
measures as :
\begin{defn}
$H_{\varphi}^{p}(\Omega)$ for $p>0$, is the space of functions
$f\in\mathcal{O}(\Omega)$ such that
\begin{equation*}
\limsup_{r\to 0^{-}}\int_{S_{\varphi,}(r)}|f|^{p}
d\mu_{\varphi,r}<\infty.
\end{equation*}
\end{defn}
The norm on these spaces is given by:
\begin{equation*}
\|f\|_{H_{\varphi}^{p}}=\left(\lim_{r\to
0^{-}}\int_{S_{\varphi}(r)}|f|^{p}
d\mu_{\varphi,r}\right)^{\frac{1}{p}}
\end{equation*}and with respect to these norm the spaces
$H_{\varphi}^{p}(\Omega)$ are Banach spaces \cite{pol}.\\
From now on we will focus on Poletsky-Stessin Hardy spaces on the
complex ellipsoids in $\mathbb{C}^{n}$ which are considered as model
cases for domains of finite type. It should be noted that although
complex ellipsoids are convex domains they are not strictly
pseudoconvex since they have Levi flat points at the boundary. The
complex ellipsoid $\mathbb{B}^{\textbf{p}}\in\mathbb{C}^{n}$ is
given as
\begin{equation*}
\mathbb{B}^{\textbf{p}}=\{z\in\mathbb{C}^{n},
\rho(z)=\sum_{j=1}^{n}|z_{j}|^{2p_{j}}-1<0\}
\end{equation*}
where $\textbf{p}=(p_{1}, p_{2},..., p_{n})\in\mathbb{Z}^{n}$. One
can easily see that
$u(z)=\log(|z_{1}|^{2p_{1}}+|z_{2}|^{2p_{2}}+...+|z_{n}|^{2p_{n}})$
is a continuous, plurisubharmonic exhaustion function for
$\mathbb{B}^{\textbf{p}}$ so we can consider the Poletsky-Stessin
Hardy spaces $H^{p}_{u}(\mathbb{B}^{\textbf{p}})$ associated with
this exhaustion function.
\section {Boundary Behavior of Poletsky-Stessin Hardy Spaces on Complex Ellipsoids}

In this section we will show that unlike the one variable case, for
$n>1$ Poletsky-Stessin Hardy spaces
$H^{p}_{u}(\mathbb{B}^{\textbf{p}})$ are not included in the
classical Hardy spaces $H^{p}(\mathbb{B}^{\textbf{p}})$ on complex
ellipsoids. Hence in this case we do not automatically inherit the
existence of boundary values from the theory of classical Hardy
spaces. Now we start with exhibiting the existence of the radial
limits for holomorphic functions in $
H^{p}_{u}(\mathbb{B}^{\textbf{p}})$, $p\geq1$.
\begin{theorem}
Let $f\in H^{p}_{u}(\mathbb{B}^{\textbf{p}})$ be a holomorphic
function. Then the radial limit function
$f^{*}(\xi)=\lim_{\tilde{r}\rightarrow 1}f(\tilde{r}\xi)$,
$\xi\in\partial\mathbb{B}^{\textbf{p}}$ exists $\mu_{u}$-almost
everywhere and $f^{*}\in
L^{p}_{\mu_{u}}(\partial\mathbb{B}^{\textbf{p}})$, $p\geq1$.
\end{theorem}
\begin{proof}
Let $\mathbb{B}^{\textbf{p}}$ be the complex ellipsoid determined by
the exhaustion function
$u(z)=\log(|z_{1}|^{2p_{1}}+|z_{2}|^{2p_{2}}+...+|z_{n}|^{2p_{n}})$
and let $\xi=(\xi_{1},\xi_{2},..,\xi{n})\in
\partial\mathbb{B}^{\textbf{p}}$, $t\in \mathbb{D}$. Suppose that $E$ is the ellipse which is the intersection of the complex
line joining $0$ to $\xi$ and the ellipsoid
$\mathbb{B}^{\textbf{p}}$. An exhaustion function for $E$ is
$g_{E}(t)=\log(A_{1}|t|^{2p_{1}}+A_{2}|t|^{2p_{2}}+..+A_{n}|t|^{2p_{n}})$
where $A_{i}=|\xi_{i}|^{2p_{i}}$, $1\leq i\leq n$. The
Monge-Amp\`ere measure associated with the exhaustion function $u$
is $d\mu_{u,r}=d^{c}u\wedge dd^{c}u|_{S_{u}(r)}$ and let $A_{0}$ be
the $n-1$-dimensional manifold of complex lines passing through the
point $0\in \mathbb{B}^{\textbf{p}}$ \cite{stout}. Now take $f\in
H^{p}_{u}(\mathbb{B}^{\textbf{p}})$ then
\begin{equation*}
\int_{S_{u}(r)}|f|^{p}d\mu_{u,r}=\int_{S_{u}(r)}|f|^{p}(d^{c}u\wedge
dd^{c}u)=\int_{A_{0}}\left(\int_{l_{z}\cap
S_{u}(r)}|f|^{p}d^{c}u\right)\omega
\end{equation*}where we have the pull-back measure $\pi^{*}\omega=dd^{c}u$ and $\pi:\bar{\mathbb{B}^{\textbf{p}}}\rightarrow
A_{0}$ is the function given by $\pi(z)=[0,z]=l_{z}$ with $l_{z}$
being the line joining $0$ and $z$.\\We can use the above
generalization of Fubini theorem since $\pi$ is a submersion and
$\pi|_{supp(d^{c}u)}$ is proper
(\cite{de3}, pg:17).\\
The measure $d^{c}u$ on $l_{z}\cap S_{u}(r)$ is equal to
$d^{c}g_{E}(t)$ on $S_{g}(r)$ and since it is a smoothly bounded
domain $d^{c}g_{E}(t)$ on $S_{g}(r)=d\mu_{g,r}$ so
\begin{equation*}
\int_{S_{u}(r)}|f|^{p}d\mu_{u,r}=\int_{A_{0}}\left(\int_{S_{g}(r)}|f|^{p}d\mu_{g,r}\right)\omega
\end{equation*} and by Fatou's lemma
$\int_{A_{0}}\left(\liminf_{r\rightarrow0}\int_{S_{g}(r)}|f|^{p}d\mu_{g,r}\right)\omega<\infty$
for $f\in H^{p}_{u}(\mathbb{B}^{\textbf{p}})$. This implies that for
$\omega$-a.e. line
$\lim_{r\rightarrow0}\int_{S_{g}(r)}|f|^{p}d\mu_{g,r}<\infty$  so
$f\in H^{p}_{g}(E)$ and it has radial boundary values
$d\sigma(\simeq d\mu_{g})$ almost everywhere \cite{stein}. Since
$f^{*}$ is the pointwise limit of measurable functions it is
measurable and consider the set
$A=\{\xi\in\partial\mathbb{B}^{\textbf{p}}, f^{*}(\xi)\quad
does\quad not \quad exist\}$, then
\begin{equation*}
\int_{\partial\mathbb{B}^{\textbf{p}}}\chi_{A}d\mu_{u}=\int_{A_{0}}\left(\int_{\partial
E}\chi_{A}(\eta)d\mu_{g}(\eta)\right)\omega.
\end{equation*} Since $f\in H^{1}_{g}(E)$, it has radial limit
values $d\mu_{g}$-a.e.  so the integral inside is $0$ and we have
$\int_{\partial\mathbb{B}^{\textbf{p}}}\chi_{A}d\mu_{u}=0$.
Therefore $f^{*}(\xi)$ exists $\mu_{u}$-a.e. Moreover for an
analytic function $f\in H^{1}_{g}(E)$ we know that the boundary
function $f^{*}\in L^{p}(\partial E)$ so we have
\begin{equation*}
\int_{\partial\mathbb{B}^{\textbf{p}}}|f^{*}|^{p}d\mu_{u}=\int_{A_{0}}\left(\int_{\partial
E}|f^{*}|^{p}d\mu_{g}\right)\omega<\infty
\end{equation*} hence $f^{*}\in
L^{p}_{\mu_{u}}(\partial\mathbb{B}^{\textbf{p}})$.
\end{proof}
Now we have two Hardy type spaces on $\mathbb{B}^{\textbf{p}}$, the
first one is the Poletsky-Stessin Hardy space
$H^{1}_{u}(\mathbb{B}^{\textbf{p}})$ and the other one is
$H^{1}(\mathbb{B}^{\textbf{p}})$ which is defined with respect to
surface area measure in accordance with Stein's definition. We will
now show that these spaces are not equal. In fact in contrast to the
one variable case Poletsky-Stessin Hardy class strictly contains the
classical Hardy space.
\begin{prop}
Let $\mathbb{B}^{\textbf{p}}$ be the complex ellipsoid. Then there
exists an exhaustion function $u$ such that
$H^{1}(\mathbb{B}^{\textbf{p}})\varsubsetneq
H^{1}_{u}(\mathbb{B}^{\textbf{p}})$.
\end{prop}
\begin{proof}
We will explicitly construct the exhaustion function $u$ by taking
$n=2$ and $\textbf{p}=(1,2)$. First of all the relation between
$d\sigma$ and $d\mu_{u}$ on $\partial\mathbb{B}^{2}$ is given by
$K_{1}|\xi_{2}|^{2}d\sigma\leq d\mu_{u}\leq
K_{2}|\xi_{2}|^{2}d\sigma$ for some $K_{1},K_{2}>0$ (depending only
on dimension and $p=(1,2)$), now consider the analytic function
$f(z_{1},z_{2})=\displaystyle{\frac{1}{(1-z_{1}^{2})^{2\alpha}}}$
where $\frac{2}{16}<\alpha<\frac{4}{16}$. We have
\begin{equation*}
\int_{\partial\mathbb{B}^{2}}|f^{*}||\xi_{2}|^{2}d\sigma=\int_{|\xi_{2}|^{4}<1}\left(\int_{|\xi_{1}|=\sqrt{1-|\xi_{2}|^{4}}}|f^{*}|d\xi_{1}\right)|\xi_{2}|^{2}d\xi_{2}
\end{equation*}
\begin{equation*}
=\int_{|\xi_{2}|^{4}<1}\left(\int_{0}^{2\pi}\frac{1}{|1-(\sqrt{1-|\xi_{2}|^{4}}e^{i\theta})^{2}|^{2\alpha}}d\theta\right)|\xi_{2}|^{2}d\xi_{2}
\end{equation*}
\begin{equation*}
=\int_{|\xi_{2}|^{4}<1}\left(\int_{0}^{2\pi}\frac{1}{|1-e^{2i\theta}+|\xi_{2}|^{4}e^{2i\theta}|^{2\alpha}}d\theta\right)|\xi_{2}|^{2}d\xi_{2}
\end{equation*}Now we will consider the behavior of the inside
integral near the point $\{1\}$ i.e. as $\theta\rightarrow0$ (this
is the only problematic point as $|\xi_{2}|\rightarrow0$).
\begin{equation*}
\lim_{\theta\rightarrow0}\frac{(1-2(1-|\xi_{2}|^{4})\cos2\theta+(1-|\xi_{2}|^{4})^{2})^{\alpha}}{|\xi_{2}|^{8\alpha}}=1
\end{equation*}so our integral becomes for $t>0$,$\delta>0$
\begin{equation*}
=\int_{|\xi_{2}|^{4}<1}\left(2\int_{t}^{\pi-t}\frac{1}{|1-e^{2i\theta}+|\xi_{2}|^{4}e^{2i\theta}|^{2\alpha}}d\theta\right)|\xi_{2}|^{2}d\xi_{2}+2\int_{B_{\delta}(0)}\frac{2t}{|\xi_{2}|^{8\alpha}}|\xi_{2}|^{2}d\xi_{2}
\end{equation*}
\begin{equation*}
+2\int_{|\xi_{2}|^{4}<1\setminus
B_{\delta}(0)}\frac{2t}{|\xi_{2}|^{8\alpha}}|\xi_{2}|^{2}d\xi_{2}
\end{equation*} since we are away from the singularity first and
third integrals are finite and if we take
$\frac{2}{16}<\alpha<\frac{4}{16}$ then second integral is also
finite and we have $f\in H^{1}_{u}(\mathbb{B}^{2})$ but $f\notin
H^{1}(\mathbb{B}^{2})$ since for this choice of $\alpha$
\begin{equation*}
\int_{|\xi_{2}|^{4}<1}\left(\int_{0}^{2\pi}\frac{1}{|1-e^{2i\theta}+|\xi_{2}|^{4}e^{2i\theta}|^{2\alpha}}d\theta\right)d\xi_{2}
\end{equation*}diverges.
\end{proof}In the previous
results we have shown that for the functions in the Poletsky-Stessin
Hardy class $H^{p}_{u}(\mathbb{B}^{\textbf{p}})$ we have the radial
limit values and throughout the following arguments we will study
the behavior of these boundary values in detail. In the classical
Hardy space theory on strictly pseudoconvex domains, Stein showed
the existence of boundary values along admissible approach regions
that are more general than the radial approach. Throughout the rest
of the section we will show that for the functions in the
Poletsky-Stessin Hardy class $H^{p}_{u}(\mathbb{B}^{\textbf{p}})$
boundary values along admissible approach regions exist. Although we
use the general idea in Stein's classical method, our approach
differs in two aspects, respectively the use of Cauchy-Fantappie
kernel instead of Poisson kernel and the use of radial limits. In
the study of the boundary behavior of holomorphic functions, having
the boundary of the domain as a space of homogenous type seems to be
a leitmotif because one of the most commonly used methods in order
to understand boundary behavior is to use maximal functions
(\cite{stein}, Theorem 3) and the natural setting for this type of
analysis is homogenous spaces. Therefore we will start with
recalling the properties of homogenous spaces and then as an
application of this classical method we will show that polynomials
are dense in the Poletsky-Stessin Hardy spaces
$H^{p}_{u}(\mathbb{B}^{\textbf{p}})$ on complex ellipsoids. Before
proceeding our arguments in $\mathbb{C}^{n}$ with maximal functions,
let us first mention the spaces of homogenous type in
$\mathbb{C}^{n}$ :
\begin{defn}
Suppose that we are given a space $X$ which is equipped with a
quasi-metric $\rho$ (\cite{krantz}, pg:145) and a regular Borel
measure $\mu$ on $X$. Denote the balls in this quasi-metric by
$B(x,r)=\{y\in X:\quad\rho(x,y)<r\}$. We say that $(X,\rho,\mu)$ is
a space of homogenous type if the following conditions are
satisfied:
\begin{itemize}
\item For each $x\in X$ and $r>0$ , $0<\mu(B(x,r))<\infty$
\item (\textbf{Doubling Condition}) There is a constant $C_{2}>0$ such that for any $x\in X$ and
$r>0$ we have $\mu(B(x,2r))\leq C_{2}\mu(B(x,r))$.
\end{itemize}
\end{defn}
Let $\Omega\subset\subset\mathbb{C}^{n}$ be a smoothly bounded
domain such that we have a quasi-metric $\rho$ on
$\overline{\Omega}$ and a regular Borel measure $\mu$ on
$\partial\Omega$. Let
$K(z,\xi):\Omega\times\partial\Omega\rightarrow\mathbb{C}$ be a
kernel such that $K(z,\xi)\in L^{1}(d\mu)$ for $z\in \Omega$,
$\xi\in\partial \Omega$. Let us consider the integral operator
determined by $K(z,\xi)$ for an $L^{p}(d\mu)$ function $f^{*}$,
\begin{equation*}
Kf^{*}(z)=\int_{\partial \Omega}f^{*}(\xi)K(z,\xi)d\mu(\xi)
\end{equation*}and define the associated maximal function as
\begin{equation*}
Mf^{*}(\xi)=\sup_{\varepsilon>0}\frac{1}{\mu(B(\xi,\varepsilon))}\int_{B(\xi,\varepsilon)}|f^{*}|d\mu
\end{equation*}
From the corresponding results in literature (see eg. \cite{stein},
Theorem 2; \cite{zyg}, chapter 14) the fundamental theorem of the
theory of singular operators which is adopted to our setting can be
stated as:
\begin{theorem}
Suppose $f^{*}\in L^{p}(d\mu_{u})$ and $1\leq p\leq\infty$ \\
(a) $\|Mf^{*}\|_{p}\leq A_{p}\|f^{*}\|_{p}$ for $1<p\leq\infty$\\
(b) The mapping $f^{*}\rightarrow Mf^{*}$ is of weak type (1-1) i.e.
$\mu_{u}\{\xi:Mf^{*}(\xi)>\alpha \}\leq
\frac{K}{\alpha}\|f^{*}\|_{1}$ if $f^{*}\in L^{1}(d\mu_{u})$.
\end{theorem}
Now we further suppose that the following conditions are satisfied:
\begin{itemize}
\item $\rho$ is a quasi-metric on $\overline{\Omega}$
\item $(\partial\Omega,\rho,\mu)$ is a space of homogenous type
\item For all $z\in \Omega$, $\xi\in \partial \Omega$ with
$\eta=\rho(z,\xi)>0$ we have
\begin{equation*}
|K(z,\xi)|\leq C\frac{1}{\mu(B(\xi,\eta))}
\end{equation*}for some $C$ independent of $\xi$ and
$\eta$ and dependence of $C$ to the point $z$ is given as in
(\cite{thomas}, (3.3)). Such a kernel is called a standard kernel.
\end{itemize}
Following the method given in (\cite{stein},Theorem 3),which was
applied for the Poisson integrals of $L^{p}$ functions, we can now
estimate the integral operator given above in this general setting :
\begin{theorem}
Suppose $Kf^{*}(z)$ is the $K(z,\xi)$-integral of an $L^{p}(d\mu)$
function $f^{*}$ where $K(z,\xi)$ satisfies the conditions given
above. Let $Q_{\alpha}(y)=\{z\in\overline{\Omega},
\rho(y,z)<\alpha\delta_{y}(z)\}$ for $y\in\partial\Omega$, $z\in
\Omega$ with $\delta_{y}(z)=\min\{\rho(z,\partial
\Omega),\rho(z,T_{y})\}$ ($T_{y}$ is the tangent plane at $y$),
$\alpha>0$, be the admissible approach region. Then
\begin{itemize}
\item When $\rho(y,z)=\varepsilon$ and $z\in Q_{\alpha}(y)$ the
following inequality holds
\begin{equation*}
|Kf^{*}(z)|\leq\tilde{A}\sum_{k=1}^{\infty}(\mu(B(y,2^{k}\varepsilon)))^{-1}\int_{B(y,2^{k}\varepsilon)}|f^{*}|d\mu
\end{equation*}
\item $\sup_{z\in Q_{\alpha}(y)}|Kf^{*}(z)|\leq\tilde{A}Mf^{*}(y)$.
\end{itemize}
\end{theorem}
\begin{proof}
Let $Kf^{*}(z)$ be the $K(z,\xi)$-integral of the $L^{p}(d\mu)$
function $f^{*}$,
\begin{equation*}
|Kf^{*}(z)|\leq \int_{\partial \Omega}|f^{*}||K(z,\xi)|d\mu(\xi)
\end{equation*}
\begin{equation*}
=\int_{\rho(\xi,y)<2\varepsilon}|f^{*}||K(z,\xi)|d\mu(\xi)+\sum_{k=2}^{\infty}\int_{2^{k-1}\varepsilon\leq\rho(\xi,y)<2^{k}\varepsilon}|f^{*}||K(z,\xi)|d\mu(\xi)
\end{equation*}first,
\begin{equation*}
\int_{\rho(\xi,y)<2\varepsilon}|f^{*}||K(z,\xi)|d\mu(\xi)\leq\frac{C}{\mu(B(y,2\varepsilon))}\int_{B(y,2\varepsilon)}|f^{*}(\xi)|d\mu(\xi)
\end{equation*}by the condition on the kernel and the construction of approach region. Similarly since
$\rho$ is a pseudometric we have $\rho(z,\xi)\geq
\tilde{C}(\rho(\xi,y)-\rho(y,z))\geq
\tilde{C}2^{k-1}\varepsilon-\tilde{C}\varepsilon\geq\tilde{\tilde{C}}2^{k-2}\varepsilon$
if $k\geq2$ whenever
$2^{k-1}\varepsilon\leq\rho(\xi,y)<2^{k}\varepsilon$ and
$\rho(z,y)=\varepsilon$, so
$|K(z,\xi)|\leq\displaystyle{\frac{2^{2k}\tilde{\tilde{C}}}{\mu(B(y,2^{k}\varepsilon))}}$.
Hence for all $k$,
\begin{equation*}
\int_{2^{k-1}\varepsilon<\rho(\xi,y)<2^{k}\varepsilon}|f^{*}||K(z,\xi)|d\mu(\xi)\leq\frac{\acute{A}_{\alpha,n}}{2^{k}\mu(B(y,2^{k}\varepsilon))}\int_{B(y,2^{k}\varepsilon)}|f^{*}(\xi)|d\mu(\xi).
\end{equation*}Upon summing in $k$ we get the first assertion and the
second inequality is an immediate consequence of the first.
\end{proof}
In \cite{thomas}, Hansson considered the boundedness of
Cauchy-Fantappie integral operator ,$H$, from
$L^{2}_{u}(\partial\mathbb{B}^{\textbf{p}})$ into
$H^{2}_{u}(\mathbb{B}^{\textbf{p}})$. In his work he applied an
operator theory result known as $T1$-Theorem and in order to use
that result he showed the homogeneity of the boundary of the complex
ellipsoid with respect to the quasi-metric $d$ and the boundary
measure $\partial\rho\wedge(\overline{\partial}\partial\rho)^{n-1}$
where the function $\rho$ is defined as
$\rho(z)=\sum_{j=1}^{n}|z_{j}|^{2p_{j}}-1$. In fact an easy
calculation shows that this measure is the boundary Monge-Amp\`ere
measure associated with the exhaustion function
$u(z)=\log(|z_{1}|^{2p_{1}}+|z_{2}|^{2p_{2}}+...+|z_{n}|^{2p_{n}})$,
$\textbf{p}=(p_{1}, p_{2},..., p_{n})\in\mathbb{Z}^{n}$ of the
complex ellipsoid $\mathbb{B}^{\textbf{p}}$. Now let
$d(\xi,z)\doteq|v(\xi,z)|+|v(z,\xi)|$ be the quasi-metric defined on
$\overline{\mathbb{B}^{\textbf{p}}}$ where
$v(\xi,z)=\langle\partial\rho(\xi),\xi-z\rangle$. Then explicitly
$v(\xi,z)=\sum_{j=1}^{n}p_{j}|\xi_{j}|^{2(p_{j}-1)}\bar{\xi_{j}}(\xi_{j}-z_{j})$
and define the boundary balls as
$B(z,\varepsilon)=\{\xi\in\partial\mathbb{B}^{\textbf{p}},d(\xi,z)<\varepsilon
\}$. It is shown that $(\partial\mathbb{B}^{\textbf{p}},d,d\mu_{u})$
is a space of homogenous type (\cite{thomas},pg:1483) and
$\displaystyle{\frac{1}{(v(\xi,z))^{n}}}$ is a standard kernel. In
the following argument we will use his homogeneity result to apply
the previous rather general procedure on the complex ellipsoid case
with the so called
Cauchy-Fantappie kernel:\\
The Cauchy-Fantappie integral (from now on we will refer as
CF-integral) of an $L^{p}(d\mu_{u})$ function $f^{*}$ is defined as
\begin{equation*}
Hf(z)=\left(\frac{1}{2\pi
i}\right)^{n}\int_{\partial\mathbb{B}^{\textbf{p}}}\frac{f^{*}(\xi)d\mu_{u}(\xi)}{(v(\xi,z))^{n}}
\end{equation*}
Before proceeding to further results let us briefly discuss the
Cauchy-Fantappie kernel. In the theory of holomorphic functions in
one variable a fundamental tool is Cauchy integral formula and in
the case of several variables one wants a suitable generalization to
Cauchy integral. One of the possible choices for the generalization
is the so called Szeg\"{o} kernel however except for a few domains
Szeg\"{o} kernel has no explicit formulation. One other choice is
the well known Bochner-Martinelli kernel but the major shortcoming
of this kernel is that it is not holomorphic in $z$ variable (For
details see \cite{range}). Contrary to Bochner-Martinelli kernel,
Cauchy-Fantappie kernel is holomorphic in $z$ hence it is a natural
generalization of Cauchy kernel to multivariable case and it has
reproducing property for the functions in the algebra
$A(\mathbb{B}^{\textbf{p}})$ (\cite{range}, Theorem 3.4). Hardy
spaces which are examined in \cite{thomas} are exactly the
Poletsky-Stessin Hardy spaces $H^{p}_{u}(\mathbb{B}^{\textbf{p}})$
that are generated by the exhaustion function $u$. At the beginning
of this section it is shown that for the functions in
$H^{p}_{u}(\mathbb{B}^{\textbf{p}})$ the boundary value function
$f^{*}\in L^{p}(d\mu_{u})$ exists so the CF-integral of $f^{*}$ is
well-defined. Now we will show that CF-integral has reproducing
property for the functions in $H^{p}_{u}(\mathbb{B}^{\textbf{p}})$:
\begin{prop}
Let $f\in H^{p}_{u}(\mathbb{B}^{\textbf{p}})$ be a holomorphic
function then
\begin{equation*}
f(z)=Hf(z)=\left(\frac{1}{2\pi
i}\right)^{n}\int_{\partial\mathbb{B}^{\textbf{p}}}\frac{f^{*}(\xi)d\mu_{u}(\xi)}{(v(\xi,z))^{n}}
\end{equation*}
\end{prop}
\begin{proof}
By the Fubini type integral formula that we used in Theorem 2.1 we
get that
\begin{equation*}
Hf(z)=\left(\frac{1}{2\pi
i}\right)^{n}\int_{A_{0}}\left(\int_{\partial
E}\frac{f^{*}(\eta)}{(v(\eta,z))^{n}}d\mu_{g}(\eta)\right)\omega
\end{equation*}and on every ellipse $E$ by (\cite{stein}, 9.7) we
have reproducing property as a consequence of one variable Cauchy
integral formula. Hence the result follows.
\end{proof}
Now define the maximal function for the functions in
$L^{p}(d\mu_{u})$ as follows :
\begin{equation*}
Mf^{*}(\xi)=\sup_{\varepsilon>0}\frac{1}{\mu_{u}(B(\xi,\varepsilon))}\int_{B(\xi,\varepsilon)}|f^{*}|d\mu_{u}
\end{equation*}The next result is a consequence of the general method given in Theorem 2.3 for complex ellipsoid case and it gives the
relation between the CF-integral and the maximal function of an
$L^{p}(d\mu_{u})$ function $f^{*}$:
\begin{cor}
Suppose $Hf(z)$ is the CF-integral of an $L^{p}(d\mu_{u})$ function
$f^{*}$. Let
$Q_{\alpha}(y)=\{z\in\overline{\mathbb{B}^{\textbf{p}}},
|v(y,z)|<\alpha\delta_{y}(z)\}$ for
$y\in\partial\mathbb{B}^{\textbf{p}}$, $z\in
\mathbb{B}^{\textbf{p}}$ with $\delta_{y}(z)=\min\{d(z,\partial
X),d(z,T_{y})\}$ ($T_{y}$ is the tangent plane at $y$), $\alpha>0$,
be the admissible approach region. Then
\begin{itemize}
\item When $d(y,z)=\varepsilon$ and $z\in Q_{\alpha}(y)$ the
following inequality holds
\begin{equation*}
|Hf(z)|\leq\tilde{A}\sum_{k=1}^{\infty}(\mu_{u}(B(y,2^{k}\varepsilon)))^{-1}\int_{B(y,2^{k}\varepsilon)}|f^{*}|d\mu_{u}
\end{equation*}
\item $\sup_{z\in Q_{\alpha}(y)}|Hf(z)|\leq\tilde{A}Mf^{*}(y)$.
\end{itemize}
\end{cor}
Next using this maximal function tools we will see the existence of
boundary values on the admissible approach regions $Q_{\alpha}(y)$,
$y\in\partial\mathbb{B}^{\textbf{p}}$:
\begin{theorem}
Let $f\in H^{p}_{u}(\mathbb{B}^{\textbf{p}})$ be a holomorphic
function and $1\leq p<\infty$. Suppose that $f^{*}$ is the radial
limit function then
\begin{equation*}
 \lim_{Q_{\alpha}(\xi)\ni
z\rightarrow\xi}f(z)=f^{*}(\xi)
\end{equation*}
exists for almost every $\xi\in\partial\mathbb{B}^{\textbf{p}}$.
\end{theorem}
\begin{proof}
If $\varepsilon>0$ then choose $g\in
C(\partial\mathbb{B}^{\textbf{p}})$ so that
$\|f^{*}-g\|_{L^{p}_{u}(\partial\mathbb{B}^{\textbf{p}})}<\varepsilon^{2}$.
Then we know that $\lim_{Q_{\alpha}(\xi)\ni
z\rightarrow\xi}Hg(z)=g(\xi)$ for all
$\xi\in\partial\mathbb{B}^{\textbf{p}}$. Therefore
\begin{equation*}
\mu_{u}\{\xi: \limsup_{Q_{\alpha}(\xi)\ni
z\rightarrow\xi}|f(z)-f^{*}(\xi)|>\varepsilon\}\leq\mu_{u}\{\xi:
\limsup_{Q_{\alpha}(\xi)\ni
z\rightarrow\xi}|f(z)-Hg(z)|>\varepsilon/3\}
\end{equation*}
\begin{equation*}
+\mu_{u}\{\xi: \limsup_{Q_{\alpha}(\xi)\ni
z\rightarrow\xi}|Hg(z)-g(\xi)|>\varepsilon/3\}+\mu_{u}\{\xi:
\limsup_{Q_{\alpha}(\xi)\ni
z\rightarrow\xi}|g(\xi)-f^{*}(\xi)|>\varepsilon/3\}
\end{equation*}

$\leq\mu_{u}\{\xi:C_{\alpha}M(f^{*}-g)
>\varepsilon/3\}+(\|f^{*}-g\|_{L^{p}_{u}(\partial \mathbb{B}^{\textbf{p}})}/(\varepsilon/3))^{p}\leq
\acute{C}_{\alpha}\varepsilon^{p} $\\ Hence the result follows.
\end{proof}
Next, we will give an invariant form of the Fatou type theorem for
the boundary values of Poletsky-Stessin Hardy spaces on complex
ellipsoids. Let us first give the preliminaries :\\
Let $k_{\mathbb{B}^{\textbf{p}}}$ be the Kobayashi-Royden metric on
$\mathbb{B}^{\textbf{p}}$. Let $U$ be a tubular neighborhood of
$\partial\mathbb{B}^{\textbf{p}}$ and take $\varepsilon_{0}$ to be
the one fourth of the distance of $U^{c}$ to
$\partial\mathbb{B}^{\textbf{p}}$. Let $\nu_{P}$ be the unit outward
normal vector to $\partial\mathbb{B}^{\textbf{p}}$ at a boundary
point $P\in\partial\mathbb{B}^{\textbf{p}}$. Take a positive
constant $\beta>0$. If $P\in\partial\mathbb{B}^{\textbf{p}}$, then
we let $n_{P}=\{P-t\nu_{P}:\quad 0<t<\varepsilon_{0}\}$. We set
\begin{equation*}
\mathcal{K}_{\beta}(P)=\{z\in\mathbb{B}^{\textbf{p}}:\quad
k_{\mathbb{B}^{\textbf{p}}}(z,n_{P})<\beta\}
\end{equation*}
We know that $\partial\mathbb{B}^{\textbf{p}}$ is strongly
pseudoconvex at all points
$z\in(\partial\mathbb{B}^{\textbf{p}})\cap(\mathbb{C}_{*})^{n}$ so
$\mu_{u}$ almost all points on $\partial\mathbb{B}^{\textbf{p}}$ are
strongly pseudoconvex points ($\ast$). Now combining Theorem 2.4
with (\cite{aladro}, Theorem 1) and using ($\ast$), we obtain the
invariant form of the Fatou type result that we proved in the
previous theorem :
\begin{theorem}
Let $f\in H^{p}_{u}(\mathbb{B}^{\textbf{p}})$ be a holomorphic
function and $1\leq p<\infty$. Suppose that $f^{*}$ is the radial
limit function then for $\beta>0$,
\begin{equation*}
\lim_{\mathcal{K}_{\beta}(P)\ni z\rightarrow P}f(z)=f^{*}(P)
\end{equation*}exists $\mu_{u}$-almost every point $P\in
\partial\mathbb{B}^{\textbf{p}}$.
\end{theorem}

As another application of the result given in Corollary 2.1, we will
show an approximation result on the Poletsky-Stessin Hardy spaces:
\begin{theorem}
Polynomials are dense in $H^{p}_{u}(\mathbb{B}^{\textbf{p}})$.
\end{theorem}
\begin{proof}
Let $f\in H^{p}_{u}(\mathbb{B}^{\textbf{p}})$ be a holomorphic
function, $1\leq p<\infty$ and let $f_{r}(\xi)=f(r\xi)$ for
$\xi\in\partial\mathbb{B}^{\textbf{p}}$. Then we have
$f(r\xi)\rightarrow f^{*}(\xi)$ $\mu_{u}$ almost everywhere. By the
previous proposition we know that $Hf(z)=f(z)$ when $f\in
H^{1}_{u}(\mathbb{B}^{\textbf{p}})$. Using this and the previous
results on maximal function we have $|f(r\xi)|\leq Mf^{*}$, where
$Mf^{*}\in L^{p}_{u}(\partial\mathbb{B}^{\textbf{p}})$ then by the
Lebesgue Dominated Convergence Theorem we have that
$f_{r}\rightarrow f^{*}$ in
$L^{p}_{u}(\partial\mathbb{B}^{\textbf{p}})$. Furthermore the
complex ellipsoid is a complete Reinhardt domain so as a consequence
of series expansion we deduce that polynomials are dense in
$A(\mathbb{B}^{\textbf{p}})$ in the topology of uniform convergence
on compact subsets (\cite{zorn}, Lemma 2). Hence polynomials are
dense in $H^{p}_{u}(\mathbb{B}^{\textbf{p}})$.
\end{proof}

\section {Composition Operators: Boundedness and Compactness}
Let $\phi:\mathbb{B}^{\textbf{p}}\rightarrow
\mathbb{B}^{\textbf{p}}$ be a holomorphic self map of
$\mathbb{B}^{\textbf{p}}$. The linear composition operator induced
by the symbol $\phi$ is defined by $C_{\phi}(f)=f\circ\phi$,
$f\in\mathcal{O}(\mathbb{B}^{\textbf{p}})$. In \cite{pol}, Poletsky
and Stessin gave necessary and sufficient conditions for the
boundedness and compactness of a composition operator acting on
Poletsky-Stessin Hardy spaces in terms of generalized Nevanlinna
counting functions. In this section we will characterize the
boundedness and compactness of composition operators acting on
Poletsky-Stessin Hardy spaces on complex ellipsoids in terms of
Carleson conditions.\\ Let $d$ be the quasi-metric given in Section
2, then given a boundary point $\xi$ and a positive constant
$\varepsilon>0$ we set the balls as follows:
\begin{equation*}
Q(\xi,\varepsilon)=\{z\in\overline{\mathbb{B}^{\textbf{p}}}:\quad
d(z,\xi)<\varepsilon\}
\end{equation*}
\begin{equation*}
B(\xi,\varepsilon)=Q(\xi,\varepsilon)\cap\partial\mathbb{B}^{\textbf{p}}
\end{equation*}
As in Section 2, for a function $f^{*}\in
L^{p}_{u}(\mathbb{B}^{\textbf{p}})$, $Hf$ denotes the
Cauchy-Fantappie integral of $f^{*}$. Now with this notation we have
the following result:
\begin{theorem}
Let $\mu$ be a positive, finite measure on
$\overline{\mathbb{B}^{\textbf{p}}}$. Then $\mu$ is bounded in
$L^{p}_{u}(\mathbb{B}^{\textbf{p}})$, $1<p<\infty$ i.e. for some
positive constant $C>0$,
\begin{equation*}
(\ast)\quad
\int_{\overline{\mathbb{B}^{\textbf{p}}}}|Hf|^{p}d\mu\leq
C\int_{\partial\mathbb{B}^{\textbf{p}}}|f^{*}|^{p}d\mu_{u}
\end{equation*}if and only if
\begin{equation*}
(\ast\ast) \quad \mu(Q(\xi,\varepsilon))\leq
C\mu_{u}(B(\xi,\varepsilon))
\end{equation*}for all $\xi$ and $\varepsilon$.
\end{theorem}
\begin{proof}
Taking $f^{*}=\chi_{Q(\xi,\varepsilon)}$ in $(\ast)$ we immediately
obtain $(\ast\ast)$. For the converse direction, let $f^{*}\in
L^{p}_{u}(\partial\mathbb{B}^{\textbf{p}})$ be given. For
$\lambda>0$, note that by lower semicontinuity the set $\{\xi\in
\partial\mathbb{B}^{\textbf{p}}:\quad Mf^{*}(\xi)>\lambda\}$ is
open, so it consists of countably many open balls $A_{j}$ on
$\partial\mathbb{B}^{\textbf{p}}$. For $\lambda$ large enough that
this set is not the whole boundary, let $Q_{j}$ be the ball
$Q(\xi,\varepsilon)$ for which $\xi$ is the center of $A_{j}$ and
the radius $\varepsilon$ is such that $d(z,\xi)=\varepsilon$
contains the boundary of $A_{j}$. If $z=r\xi_{0}$,
$\xi_{0}\in\partial\mathbb{B}^{\textbf{p}}$, and $|Hf(z)|>\lambda$,
let $J_{z}$ be the set
$\{\eta\in\partial\mathbb{B}^{\textbf{p}}:\quad z\in Q_{3}(\eta)\}$
where $Q_{\alpha}(\xi)$ is the admissible approach region defined in
Section 2. The point $\xi_{0}$ is the center of $J_{z}$ and if
$\gamma$ is a boundary point of $J_{z}$, from the definition of
$Q_{3}(\eta)$ we have,
\begin{equation*}
d(z,\gamma)=3\delta_{\gamma}(z)=3d(z,\xi_{0})
\end{equation*}Thus,
\begin{equation*}
d(\xi_{0},\gamma)\geq C(d(z,\gamma)-d(z,\xi_{0}))=2Cd(z,\xi_{0})
\end{equation*}which means that $z\in Q(\xi_{0},d(\xi_{0},\gamma))$.
Now $z\in Q_{3}(\eta)$ implies $Mf^{*}(\eta)>\lambda$ for $\eta\in
J_{z}$. This means $J_{z}$ is contained in $A_{j}$ for some $j$ so $
Q(\xi_{0},d(\xi_{0},\gamma))$ is a subset of $Q_{j}$ and $z$ is in
$Q_{j}$. Now by $(\ast\ast)$ for some constant $K>0$,
\begin{equation*}
\mu(\{z:\quad|Hf(z)|>\lambda\})\leq\sum\mu(Q_{j})\leq
K\sum\mu_{u}(A_{j})=K\mu_{u}(\{\eta:\quad Mf^{*}(\eta)>\lambda\})
\end{equation*}Now we have that,
\begin{equation*}
\int_{\mathbb{B}^{\textbf{p}}}|Hf(z)|^{p}d\mu=\int_{0}^{\infty}p\lambda^{p-1}\mu(\{z:\quad|Hf(z)|>\lambda\})d\lambda
\end{equation*}
\begin{equation*}
\leq K\int_{0}^{\infty}p\lambda^{p-1}\mu_{u}(\{\eta:\quad
Mf^{*}(\eta)>\lambda\})d\mu_{u}(\eta).
\end{equation*}Then by the maximal function result (\cite{sah}, Corollary
3.3.1) we obtain that,
\begin{equation*}
\int_{\mathbb{B}^{\textbf{p}}}|Hf(z)|^{p}d\mu\leq\widetilde{C}\|f^{*}\|_{L^{p}_{u}(\partial\mathbb{B}^{\textbf{p}})}.
\end{equation*}
\end{proof}As a consequence of this result, one can deduce the
following characterization for the boundedness of the composition
operators:
\begin{theorem}
Let $\phi:\mathbb{B}^{\textbf{p}}\rightarrow
\mathbb{B}^{\textbf{p}}$ be a holomorphic self map of
$\mathbb{B}^{\textbf{p}}$. For $1\leq p<\infty$, the composition
operator $C_{\phi}(f)=f\circ\phi$ is bounded on
$H^{p}_{u}(\mathbb{B}^{\textbf{p}})$ if and only if
$\mu(Q(\xi,\varepsilon))\leq C\mu_{u}(B(\xi,\varepsilon))$ for all
$\xi\in\partial\mathbb{B}^{\textbf{p}}$ and $\varepsilon>0$ where
$\mu(E)=\mu_{u}((\phi)^{-1}(E))$ for all measurable
$E\subset\overline{\mathbb{B}^{\textbf{p}}}$.
\end{theorem}Lastly, we will give necessary and sufficient
conditions for the compactness of composition operators acting on
$H^{p}_{u}(\mathbb{B}^{\textbf{p}})$ :
\begin{theorem}
The composition operator $C_{\phi}(f)=f\circ\phi$ is compact on
$H^{p}_{u}(\mathbb{B}^{\textbf{p}})$ if and only if
$\mu(Q(\xi,\varepsilon))=o(\mu_{u}(B(\xi,\varepsilon)))$ as
$\varepsilon\rightarrow0$ uniformly on
$\xi\in\partial\mathbb{B}^{\textbf{p}}$.
\end{theorem}
\begin{proof}
Assume for a contradiction that $\mu(Q(\xi,\varepsilon))\neq
o(\mu_{u}(B(\xi,\varepsilon)))$ so that we can find
$\xi_{n}\in\partial\mathbb{B}^{\textbf{p}}$, positive numbers
$h_{n}$ decreasing to $0$ and $\beta>0$ with
$\mu(Q(\xi_{n},h_{n}))\geq\beta\mu_{u}(B(\xi_{n},h_{n}))$. Set
$a_{n}=(1-h_{n})\xi_{n}$ and define $f_{n}$ so that
$f_{n}(z)=(1-\overline{a_{n}}z)^{-\frac{4}{p}}$ then by the Fubini
type formula that we obtained in the proof of Theorem 2.1 we see
that,
\begin{equation*}
\|f_{n}\|^{p}_{p}=\int_{\partial\mathbb{B}^{\textbf{p}}}|f_{n}|^{p}d\mu_{u}=\int_{A_{0}}\left(\int_{E}|f_{n}|^{p}d\mu_{g}\right)\omega.
\end{equation*}Now for all $\xi\in\partial\mathbb{B}^{\textbf{p}}$,
in the inner integral we make the change of variables
$\gamma(\xi_{n}^{j})=(\xi_{n}^{j})^{p_{j}}$ then if we define
$\tilde{f}(\xi^{p})=f(\xi)$we obtain that
\begin{equation*}
\int_{E}|f_{n}|^{p}d\mu_{g}\sim\int_{0}^{2\pi}|\tilde{f}\mid_{\mathbb{T}}|^{p}d\theta\sim(1-\|a_{n}\|^{2})^{-3n}\sim
h_{n}^{-3n}
\end{equation*} so we have $\|f_{n}\|^{p}_{p}\sim h_{n}^{-3n}$. Thus
if $g_{n}=\displaystyle\frac{f_{n}}{\|f_{n}\|_{p}}\in
H^{p}_{u}(\mathbb{B}^{\textbf{p}})$, we have $g_{n}$ converges to
$0$ weakly since $h_{n}\rightarrow0$ as $n\rightarrow\infty$.
However,
\begin{equation*}
\|g_{n}\circ\phi\|^{p}_{p}=\int_{\partial\mathbb{B}^{\textbf{p}}}|g_{n}\circ\phi|^{p}d\mu_{u}=\int_{\overline{\mathbb{B}^{\textbf{p}}}}|g_{n}|^{p}d\mu
\end{equation*}
\begin{equation*}
\geq\|f_{n}\|^{-p}_{p}\int_{Q(\xi_{n},h_{n})}|f_{n}|^{p}d\mu.
\end{equation*}
If $z\in Q(\xi_{n},h_{n})$, then since $d(z,\xi_{n})\leq h_{n}$
implies $|\xi_{n}-z|\leq h_{n}$ and
\begin{equation*}
|1-\overline{a_{n}}z|=|1-(1-h_{n})\overline{\xi_{n}}z|\leq|\overline{\xi_{n}}(\xi_{n}-z)|+|h_{n}\overline{\xi_{n}}z|\leq2h_{n}
\end{equation*}we have $|f_{n}|^{p}\geq(2h_{n})^{-4}$ on
$Q(\xi_{n},h_{n})$. Thus $\|g_{n}\circ\phi\|^{p}_{p}$ is bounded
away from $0$ and $C_{\phi}$ cannot be compact.\\
For the converse direction assume
$\mu(Q(\xi,\varepsilon))=o(\mu_{u}(B(\xi,\varepsilon)))$ uniformly
in $\xi$. Then given $\varepsilon>0$, there exists a $\delta_{0}>0$
so that for all $\delta<\delta_{0}$ we have that
\begin{equation*}
\mu(Q(\xi,\delta))\leq
2\varepsilon\mu_{u}(B(\xi,\delta))\quad(\ast\ast\ast)
\end{equation*}
Now suppose $\{f_{n}\}$ is a bounded sequence in
$H^{p}_{u}(\mathbb{B}^{\textbf{p}})$ and $f_{n}\rightarrow f$
uniformly on compact subsets of $\mathbb{B}^{\textbf{p}}$. Then,
\begin{equation*}
\int_{\partial\mathbb{B}^{\textbf{p}}}|(f_{n}-f)\circ\phi|^{p}d\mu_{u}=\int_{\overline{\mathbb{B}^{\textbf{p}}}}|f_{n}-f|^{p}d\mu.
\end{equation*}Now decompose $\mu$ so that $\mu=\mu_{1}+\mu_{2}$
where $\mu_{1}$ is the restriction of $\mu$ to
$(1-\delta_{0})\overline{\mathbb{B}^{\textbf{p}}}$ and
$\mu_{2}=\mu-\mu_{1}$. Then,
\begin{equation}\label{eq:decompose}
\int_{\overline{\mathbb{B}^{\textbf{p}}}}|f_{n}-f|^{p}d\mu=\int_{\overline{\mathbb{B}^{\textbf{p}}}}|f_{n}-f|^{p}d\mu_{1}+\int_{\overline{\mathbb{B}^{\textbf{p}}}}|f_{n}-f|^{p}d\mu_{2}.
\end{equation}Since $\mu_{2}<\mu$, $\mu_{2}$ satisfies
$(\ast\ast\ast)$ whenever $\mu$ does. We claim that $\mu_{2}$
satisfies the condition $(\ast\ast)$ in the previous theorem. To see
this claim, note that if $Q(\xi,\eta)\subset
N=\overline{\mathbb{B}^{\textbf{p}}}\setminus(1-\delta_{0})\overline{\mathbb{B}^{\textbf{p}}}$,
the claim is immediate from $(\ast\ast\ast)$. For an arbitrary
$B(\xi,\eta)$ decompose $B(\xi,\eta)$ into a union of open balls
$B(\xi_{j},\delta_{j})$ so that
$\mu_{u}(B(\xi_{j},\delta_{j}))<\delta_{0}$ and
$\sum_{j}\mu_{u}(B(\xi_{j},\delta_{j}))\leq2\mu_{u}B(\xi,\delta)$.
Then $Q(\xi_{j},\delta_{j})\subset N$ and $Q(\xi,\delta)\cap
N=\bigcup_{j}Q(\xi_{j},\delta_{j})$. Hence,
\begin{equation*}
\mu_{2}(Q(\xi,\delta))=\mu(Q(\xi,\delta)\cap
N)\leq\sum_{j}\mu_{u}(Q(\xi_{j},\delta_{j}))
\end{equation*}
\begin{equation*}
\leq2\varepsilon\mu_{u}B(\xi_{j},\delta_{j})\leq4\varepsilon\mu_{u}B(\xi,\delta).
\end{equation*}Therefore,
$\mu_{2}(Q(\xi,\delta))\leq4\varepsilon\mu_{u}B(\xi,\delta)$ for all
$B(\xi,\delta)$ and thus in the equation (\ref{eq:decompose}) the
first integral can be made arbitrarily small by choosing $n$
sufficiently large and for the second integral we have that
\begin{equation*}
\int_{\overline{\mathbb{B}^{\textbf{p}}}}|f_{n}-f|^{p}d\mu_{2}\leq
C\varepsilon\|f_{n}-f\|_{p}
\end{equation*}for some constant $C$ by the previous theorem so
$\varepsilon$ can be chosen arbitrarily small and the result
follows.

\end{proof}

\section* {Acknowledgments}
This paper is based on the second part of my PhD dissertation and I
would like to express my sincere gratitude to my advisor
Prof.Ayd{\i}n Aytuna for his valuable guidance and support. I would
also like to thank Prof. Evgeny A.Poletsky for the valuable
discussions that we had during my visit at Syracuse University.


\vspace*{20pt}
\begin{thebibliography}{99}
\bibitem{aladro} Gerardo Aladro, \textit{The Comparability of the Kobayashi Approach Region and the Admissible Approach
Region}, Illinois Journal of Mathematics, Vol: 33, No:1, (1989).


\bibitem{de1} Jean-Pierre Demailly, \textit{Mesures de Monge-Amp\`ere
et Caract\'erisation G\'eom\'etrique des Vari\'et\'es Alg\'ebraiques
Affines}, M\'emoire de la Soci\'et\'e Math\'ematique de France
\textbf{19},1-124, (1985).
\bibitem{de2} Jean-Pierre Demailly, \textit{Mesures de
Monge-Amp\`ere et Mesures Pluriharmoniques}, Matematische
Zeitschrift, No:194, 519-564, (1987).
\bibitem{de3} Jean-Pierre Demailly, \textit{Complex Analytic and
Differential Geometry} , unpublished manuscript.



\bibitem{thomas} Thomas Hansson, \textit{On Hardy Spaces in Complex Ellipsoids
}, Annales de l'institut Fourier 49, 1477-1501, (1999).

\bibitem{krantz} Steven G. Krantz, \textit{Fatou Theorems Old and New:
An Overview of The Boundary Behavior of Holomorphic Functions}





\bibitem{pol} Evgeny A. Poletsky,Michael I. Stessin, \textit{Hardy
and Bergman Spaces on Hyperconvex Domains and Their Composition
Operators} Indiana Univ. Math. J. \textbf{57}, 2153-2201, (2008).



\bibitem{range} R. Michael Range, \textit{Holomorphic Functions and Integral Representations in Several Complex
Variables}, Springer-Verlag New York Inc., (1986).



\bibitem{sah} Sibel \c{S}ahin, \textit{Monge-Amp\`ere Measures and Poletsky-Stessin Hardy
Spaces on Bounded Hyperconvex Domains}, PhD Dissertation, Sabanc{\i}
University, 2014.

\bibitem{stein} Elias M. Stein, \textit{The Boundary Behavior of
Holomorphic Functions of Several Complex Variables}, Princeton
University Press, Princeton, (1972).
\bibitem{stout} Edgar Lee Stout, \textit{The Boundary Values of Holomorphic Functions of Several Complex
Variables}, Duke Mathematical Journal, volume 44, no:1, 105-108,
(1977).
\bibitem{zorn} Paul Zorn, \textit{Analytic Functionals and The Bergman Projection on Circular
Domains}, Proceedings of the American Mathematical Society, Vol:96,
No:3, (1986).

\bibitem{zyg} A. Zygmund, \textit{Trigonometric Series}, Cambridge
University Press, Third Edition Volumes I-II Combined , (2002).




\end{thebibliography}
\end{document}